\documentclass{article}

\usepackage[latin1]{inputenc}
\usepackage[T1]{fontenc}
\usepackage{enumerate}
\usepackage[all]{xy}
\usepackage{amsmath,amsthm,amsfonts,amssymb,stmaryrd}
\usepackage[pdftex,bookmarks=false]{hyperref}

\newtheorem{lem}{Lemma}

\newtheorem{thm}[lem]{Theorem}
\newtheorem*{thm*}{Theorem}
\newtheorem*{thmzp*}{$\mathbf{Z^*_p}$-theorem}
\newtheorem*{conjze*}{$\mathbf{Z^*_e}$-conjecture}

\theoremstyle{definition}

\newtheorem{hyp}[lem]{Assumption}

\newcommand\catname[1]{\mathbf{#1}}

\newcommand\catalg{\catname{Alg}}

\newcommand\catmod[1]{{\null_{#1}\catname{Mod}}}

\newcommand\catf{\catname{F}}
\newcommand\catbr{\catname{Br}}
\newcommand\catfr{\catname{Fr}}

\DeclareMathOperator{\Autom}{Aut}
\DeclareMathOperator{\Endom}{End}
\DeclareMathOperator{\Res}{Res}
\DeclareMathOperator{\Inf}{Inf}
\DeclareMathOperator{\Defres}{Defres}

\DeclareMathOperator{\Ind}{Ind}
\DeclareMathOperator{\Br}{Br}
\DeclareMathOperator{\br}{br}

\DeclareMathOperator{\Tr}{Tr}

\newcommand\longisom{\tilde{\,\longrightarrow}\,}
\newcommand\op{\mathrm{op}}
\newcommand\bop{\mathrm{o}}
\newcommand\ringO{\mathcal O}

\newcommand\dade{\mathcal D}

\newcommand\sla[1]{\!\left<#1\right>}

\title{Strong fusion control and stable equivalences}
\author{Erwan Biland%
\footnote{Permanent e-mail adress: \texttt{eb@erwanbiland.fr}}}

\begin{document}

\maketitle

\begin{abstract}
This article is dedicated to the proof of the following theorem.
Let $G$ be a finite group, $p$ be a prime number, and $e$ be a $p$-block of $G$. Assume that the centraliser $C_G(P)$ of an $e$-subpair $(P,e_P)$ ``strongly'' controls the fusion of the block $e$, and that a defect group of $e$ is either abelian or (for odd~$p$) has a non-cyclic center. Then there exists a stable equivalence of Morita type between the block algebras $\ringO Ge$ and $\ringO C_G(P)e_P$, where $\ringO$ is a complete discrete valuation ring of residual characteristic $p$. This stable equivalence is constructed by gluing together a family of local Morita equivalences, which are induced by bimodules with fusion-stable endo-permutation sources. 

Broué had previously obtained a similar result for principal blocks, in relation with the search for a modular proof of the odd $Z^*_p$-theorem. Thus our theorem points towards a block-theoretic analogue of the $Z^*_p$-theorem, which we state in terms of fusion control and Morita equivalences. 
\end{abstract}

\tableofcontents

\newpage

\section*{Introduction}

Let $G$ be a finite group, $p$ be a prime number, and $P$ be a $p$-subgroup of $G$. Assume that the centraliser $C_G(P)$ controls the $p$-fusion in $G$, \emph{i.e.}, the subgroup $C_G(P)$ contains a Sylow $p$-subgroup $D$ of $G$ and $N_G(Q)\leqslant C_G(Q)C_G(P)$ for any subgroup $Q$ of $D$. Then the famous $Z^*_p$-theorem asserts that the group $G$ admits the factorisation $G=O_{p'}(G) C_G(P)$, where $O_{p'}(G)$ is the largest normal subgroup of $G$ with order coprime to $p$. This theorem has been proven originally for $p=2$ by Glauberman \cite{Glauberman1966}, and later deduced, for $p$ odd, from the classification of finite simple groups (\cite[Theorem 1]{Artemovich}, \cite[Remark 7.8.3]{Classification3}).

Let $\ringO$ be a complete discrete valuation ring with algebraically closed residue field $k$ of characteristic $p$. It is well-known, and elementary, that the factorisation $G=O_{p'}(G)\, C_G(P)$ is satisfied if, and only if, the restriction functor $\Res^{\,G}_{C_G(P)}$ induces a Morita equivalence between the principal blocks of the groups $G$ and $C_G(P)$ over the ring $\ringO$. Therefore, the $Z^*_p$ theorem can be stated in terms of Morita equivalences, and one should expect it to admit a ``modular proof'', relying on local representation theory. Such a proof is known for $p=2$ but, as of today, not for odd~$p$.

The investigation of a putative minimal counter-example to the $Z^*_p$-theorem leads to a finite group $G$ and a $p$-subgroup $P$ such that the centraliser $C_G(P)$ ``strongly'' controls the $p$-fusion in $G$, \emph{i.e.}, $C_G(P)$ contains a Sylow $p$-subgroup $D$ of $G$, and $N_G(Q) \leqslant O_{p'}(C_G(Q))C_G(P)$ for any non-trivial $p$-subgroup $Q$ of $D$. In this context, Brou\'e has proven that the restriction functor $\Res^{\,G}_{C_G(P)}$ induces a stable equivalence between the principal blocks of the groups $G$ and $C_G(P)$. His proof relies on the following statement, which appears in \cite[Theorem 5.6]{Rouquier2001}: for any two finite groups $G$ and $H$ with the same local structure, a $p$-permutation bimodule that induces Morita equivalences between the principal block algebras of the ``local'' subgroups of $G$ and $H$ must induce a stable equivalence between the principal block algebras of $G$ and $H$ themselves. 

On the one hand, proving the $Z^*_p$-theorem amounts to proving that Brou\'e's stable equivalence is actually a Morita equivalence.
On the other hand, quoting \cite{KessarLinckelmann2002}, ``it seems to be a general intuition that there should be some block-theoretic analogue of Glauberman's $Z^*$-theorem''. Such an analogue could be stated as follows.

\begin{conjze*}
Let $G$ be a group, $e$ be a block of the group algebra $\ringO G$, and $(P,e_P)$ be an $e$-subpair of the group $G$. Assume that the centraliser $C_G(P)$ controls the $e$-fusion in $G$ with respect to a maximal $e$-subpair $(D,e_D)$ that contains $(P,e_P)$.
Then there exists a Morita equivalence between the block algebras $\ringO Ge$ and $\ringO C_G(P)e_P$.
\end{conjze*}

The main theorem of this article is the following generalisation of Brou\'e's stable equivalence to the context of the above $Z^*_e$-conjecture, for a non-principal block $e$.

\begin{thm}
\label{thm:stable equivalence}
Let $G$ be a group, $e$ be a block of the group algebra $\ringO G$, and $(P,e_P)$ be an $e$-subpair of the group $G$. Assume that the centraliser $C_G(P)$ ``strongly controls the $e$-fusion in $G$'' with respect to a maximal $e$-subpair $(D,e_D)$ that contains $(P,e_P)$. Assume moreover that the defect group $D$ is abelian, or that $p$ is odd and the poset $\mathcal A_{\geqslant 2}(D)$ of elementary abelian subgroups of $D$ of rank at least 2 is connected (\emph{e.g.}, the center $Z(D)$ is non-cyclic). Then there exists a stable equivalence of Morita type between the block algebras $\ringO Ge$ and $\ringO C_G(P)e_P$.
\end{thm}

Let us sketch our proof of this theorem. Although we have two blocks $e$ and $e_P$ with the same local structure, and a family of Morita equivalences between the local block algebras attached to these blocks, we cannot use the theorem of Rouquier quoted above. Indeed, the bimodules that define those local Morita equivalences are not $p$-permutation bimodules; they admit non-trivial endo-permutation sources. Moreover, we have no given bimodule at the ``global'' level that would induce those local bimodules. Thus we need to construct the global bimodule by a gluing procedure, which roughly follows the method initiated by Puig in \cite{Puig1991}. 

To complete this task, we use the language of Brauer-friendly modules, as defined in \cite{Biland2013}. For the reader's convenience, we gather in the first two sections the definitions and results that are needed in the present article. In Section \ref{sec:CentControl}, we specialise these tools to the situation where the centraliser of a $p$-subgroup controls the fusion.

With the assumptions of Theorem \ref{thm:stable equivalence}, we have, for any non-trivial subgroup $Q$ of the defect group $D$, the following ``local'' situation: $e_Q$ is a block of a group $G_Q$ that factorises as $G_Q=O_{p'}(G_Q)C_{G_Q}(P)$ so, by  \cite{KuelshammerRobinson}, there is a Morita equivalence $k G_Q\bar e_Q \sim k C_{G_Q}(P)\br_P(\bar e_Q)$. In Section \ref{sec:local case}, we prove an equivariant version of this Morita equivalence and give a new construction of the Brauer-friendly module $M_Q$ that induces it. Moreover, we identify a vertex subpair of $M_Q$ and provide an explicit description of its source $V_Q$.

The heart of our proof is the definition of a ``global'' source $V$ from the family of local sources $(V_Q)_{1\neq Q\leqslant D}$, which is achieved in Section \ref{sec:SourceGlue}. In the non-abelian defect case, this is an application of the main theorem of \cite{BoucThevenaz}; the obstruction group that appears in this theorem explains the technical condition on the defect group $D$ that we require in Theorem \ref{thm:stable equivalence}. We hope that this technical condition can be lifted in the future.
Finally, in Section \ref{sec:GlobalCase}, we consider the unique indecomposable $(\ringO Ge,\ringO C_G(P)e_P)$-bimodule $M$ with vertex subpair $(\Delta D,e_D\otimes e_D^\bop)$ and source $V$ such that the slashed module $M\sla{\Delta P, e_P\otimes e_P^\bop}$ is isomorphic to the block algebra $kC_G(P)\bar e_P$, and we use the main result of \cite{LinckelmannUnpublished} to prove that $M$ defines a stable equivalence between the blocks $e$ and~$e_P$.

%=================
\section{General definitions and notations}
\label{sec:GenDef}

We let $\ringO$ be a complete discrete valuation ring with maximal ideal $\mathfrak p$ and algebraically closed residue field $k$ of characteristic $p$. This includes the case $\ringO=k$, so that every result that is proven over the ring $\ringO$ remains true over the field~$k$. 

For any finite group $G$, we denote by $\Delta G =\{(g,g); g\in G\}$ the diagonal subgroup of the direct product $G\times G$. We denote by $O_{p'}(G)$ the largest normal subgroup of $G$ with order coprime to $p$. For an element $g\in G$ and an object $X$, the notation $^gX$ stands for the object $gXg^{-1}$ whenever this makes sense. Let $e$ be a block of the group $G$, \emph{i.e.}, a primitive central idempotent of the group algebra $\ringO G$. We denote by $\bar e\in kG$ its reduction modulo $\mathfrak p$, and by $e^\bop$ its image by the isomorphism $(\ringO G)^\op\to \ringO G, g\mapsto g^{-1}$. For any two groups $G$ and $H$, we may consider an $(\ringO G,\ringO H)$-bimodule $M$ as an $\ringO(G\times H)$-module. If $f$ is a block of the group $H$ such that $eMf=M$, then the $\ringO(G\times H)$-module $M$ belongs to the block $e\otimes f^\bop$, where we have implicitely identified the algebras $\ringO(G\times H)$ and $\ringO G\otimes_\ringO \ringO H$ via the natural isomorphism.

Let $G$ be a finite group and $S$ be a normal subgroup of $G$. An $S$-interior $G$-algebra over the ring $\ringO$ is a triple $(A,\gamma,\iota)$, where $A$ is an $\ringO$-algebra and $\gamma:G\to\Autom_\catalg(A)$, $\iota : S\to A^\times$ are group morphisms such that, for any $s\in S$, $g\in G$ and $a\in A$,
\[
\gamma(s)(a) = \iota(s).a.\iota(s)^{-1}
\quad \text{ and } \quad
\iota(gsg^{-1}) = \gamma(g)(\iota(s)).
\]
With these notations, $A$ has a natural structure of $\ringO(S\times S)\Delta G$-module.
Let $H$ a subgroup of $G$, and $T$ be a normal subgroup of $H$ contained in $S$. Let $B$ be a $T$-interior $H$-algebra, hence an $\ringO (T\times T)\Delta H$-module. Then the induced module $\Ind_{(T\times T)\Delta H}^{(S\times S)\Delta G} B$ has a natural structure of $S$-interior $G$-algebra (\emph{cf.} \cite{Puig2002} for details about partly interior algebras). 

For instance, let $T$ be a normal subgroup of a finite group $G$, and $S$ be any normal subgroup of $G$ that contains $T$. Let $b$ be a block of the group $T$, let $H=G_b$ be the stabiliser of $b$ in $G$, and $b'=\Tr_{G_b}^G(b)$ be the sum of all $G$-conjugates of $b$. Then the block algebra $\ringO Tb$ is naturally a $T$-interior $G_b$-algebras, via the map $\iota:t\mapsto tb$ and the conjugation action of $G_b$. Moreover the interior structure map $\iota:S\to \bigl(\Ind_{(T\times T)\Delta G_b}^{(S\times S)\Delta G} \ringO Tb\bigr)^\times$ induces an isomorphism of $S$-interior $G$-algebras
\[
\ringO S b' \ \simeq \ \Ind_{(T\times T)\Delta G_b}^{(S\times S)\Delta G} \ringO Tb.
\]

Let $P$ be a $p$-subgroup of a finite group $G$. For any $\ringO G$-module $M$, we denote by $\Br_P(M)$ the Brauer quotient of $M$, \emph{i.e.}, the $k N_G(P)$-module
\[
\Br_P(M) \ = \ M^P \Bigl/ \Bigl(\sum_{Q<P}\Tr_Q^P(M^Q) +\mathfrak p M^P\Bigr),
\]
where $M^P$ is the submodule of $P$-fixed points in $M$, $\Tr_Q^P:M^Q\to M^P$ is the relative trace map (as defined, \emph{e.g.}, in \cite{AlperinBroue}). We denote by $\br_P^M:M^P\to \Br_P(M)$ the projection map. Any morphism of $\ringO G$-modules $u:L\to M$ induces a morphism of $kN_G(P)$-modules $\Br_P(u):\Br_P(L)\to\Br_P(M)$. This defines a functor
\[
\Br_P : \catmod{\ringO G} \ \to\ \catmod{k N_G(P)}.
\]
Notice that we write the Brauer functor $\Br_P$ with a capital B, and the Brauer map $\br_P$ with a lowercase b. If $A$ is a $G$-interior algebra (\emph{e.g.}, $A=\Endom_\ringO(M)$ for some $\ringO G$-module $M$), then the Brauer quotient $\Br_P(A)$ has a natural structure of $C_G(P)$-interior $N_G(P)$-algebra over the field $k$.

%=================
\section{Brauer-friendly modules and the slash construction}
\label{sec:BFmods}

This section gathers definitions and results from \cite{AlperinBroue}, \cite{Sibley1990} and \cite{Biland2013}. Notice that the latter reference uses a functorial approach that we do not need here.

Let $G$ be a finite group. The Frobenius category $\catfr(G)$ is defined as follows: an object is a $p$-subgroup; an arrow $\phi:P\to Q$ is a group morphism that is induced by an inner automorphism of the group $G$. Let $P$ be a $p$-subgroup of $G$, and let $V$ be an indecomposable $\ringO P$-module that is capped, \emph{i.e.}, with vertex $P$. We say that $(P,V)$ is a fusion-stable endo-permutation source pair if, for any $p$-subgroup $Q$ of $G$ and any two arrows $\phi_1,\phi_2:Q\to P$ in the category $\catfr(G)$, the direct sum $\Res_{\phi_1} V\oplus \Res_{\phi_2} V$ is an endo-permutation $\ringO Q$-module (\emph{i.e.}, the restrictions $\Res_{\phi_1} V$ and $\Res_{\phi_2} V$ are compatible endo-permutation $\ringO Q$-modules). Let $M$ be an indecomposable $\ringO G$-module with vertex $P$ and source $V$. We know from \cite[Theorem 1.5]{Urfer2007} that $M$ is an endo-$p$-permutation $\ringO G$-module if, and only if, the source pair $(P,V)$ is a fusion-stable endo-permutation source pair.

These ideas admit the following generalisation to blocks. Let $e$ be a block of $G$. A subpair of the group $G$ is a pair $(P,e_P)$, where $P$ is a $p$-subgroup of $G$ and $e_P$ is a block of the group $C_G(P)$. The idempotent $e_P$ is actually a block of the group $H$ whenever $H$ is a subgroup of $G$ such that $C_G(P)\leqslant H\leqslant N_G(P,e_P)$.
The subpair $(P,e_P)$ is an $e$-subpair if $\bar e_P\br_P(e)\neq0$, where $\br_P:(\ringO G)^P\to kC_G(P)$ denotes the Brauer morphism. Let $(P,e_P)$ and $(Q,e_Q)$ be two $e$-subpairs of $G$. One writes $(P,e_P)\trianglelefteqslant (Q,e_Q)$ if $(Q,e_Q)$ is an $e_P$-subpair of the group $N_G(P,e_P)$ such that $P\leqslant Q$. The antisymmetric relation $\trianglelefteqslant$ generates an order on the set of $e$-subpairs, and the group $G$ acts by conjugation on the resulting poset. The Brauer category $\catbr(G,e)$ is defined as follows: an object is an $e$-subpair $(P,e_P)$; an arrow $\phi:(P,e_P)\to (Q,e_Q)$ is a group morphism $\phi:P\to Q$ of the form $x\mapsto {^gx}$ for some element $g\in G$ such that $^g(P,e_P) \leqslant (Q,e_Q)$. This category is equivalent to a fusion system $\catf$ of the block $e$, as defined in \cite{AKO}.

Let $M$ be an indecomposable $\ringO Ge$-module, and $P$ be a vertex of $M$. Let $M'$ be an indecomposable $\ringO N_G(P)$-module that is a Green correspondent of $M$, and $f$ be the block of $N_G(P)$ such that $fM\neq 0$. Let $e_P$ be a block of $C_G(P)$ such that $fe_P\neq 0$. The subpair $(P,e_P)$ is called a vertex subpair of the indecomposable module $M$. It follows from Nagao's theorem that $(P,e_P)$ is an $e$-subpair of the group $G$. Any source $V$ of the indecomposable $\ringO N_G(P,e_P)$-module $L=e_PM'$ with respect to the vertex $P$ is called a source of $M$ with respect to the vertex subpair $(P,e_P)$. A source triple $(P,e_P,V)$ of $M$ is well-defined up to conjugation in the group $G$. 

Let $(P,e_P)$ be an $e$-subpair of the group $G$, and let $V$ be a capped indecomposable $\ringO P$-module. We say that $(P,e_P,V)$ is a fusion-stable endo-permutation source triple if, for any $e$-subpair $(Q,e_Q)$ and any two arrows $\phi_1,\phi_2:(Q,e_Q)\to (P,e_P)$ in the Brauer category $\catbr(G,e)$, the direct sum $\Res_{\phi_1} V\oplus \Res_{\phi_2} V$ is an endo-permutation $\ringO Q$-module.
We say that two fusion-stable endo-permutation source triples $(P_1,e_1,V_1)$ and $(P_2,e_2,V_2)$ are compatible if, for any $e$-subpair $(Q,e_Q)$ and any two arrows $\phi_1:(Q,e_Q)\to (P_1,e_1)$, $\phi_2:(Q,e_Q)\to (P_2,e_2)$ in the Brauer category $\catbr(G,e)$, the direct sum $\Res_{\phi_1} V_1 \oplus \Res_{\phi_2} V_2$ is an endo-permutation $\ringO Q$-module. 

We say that an $\ringO Ge$-module $M$ is Brauer-friendly if it is a direct sum of indecomposable $\ringO Ge$-modules with compatible fusion-stable endo-permutation source triples. The following two lemmas are straightforward from \cite[Lemma~9, Theorem~15, and proof of Lemma~18]{Biland2013}.

\begin{lem}
\label{lem:BFModSubpair}
Let $M$ be a Brauer-friendly $\ringO Ge$-module, and $(P,e_P)$ be an $e$-subpair of the group $G$. Any capped indecomposable direct summand of the $\ringO P$-module $e_P M$ is an endo-permutation $\ringO P$-module, and there is at most one isomorphism class of such $\ringO P$-modules.
\end{lem}

\begin{lem}
\label{lem:SlashMod}
Let $M$ be a Brauer-friendly $\ringO Ge$-module. Let $(P,e_P)$ be an $e$-subpair of the group $G$, and $H$ be a subgroup of $G$ such that $PC_G(P)\leqslant H\leqslant N_G(P,e_P)$.
\begin{enumerate}[(i)]
\item There exists a Brauer-friendly $kH\bar e_P$-module $M_0$ and an isomorphism of $C_G(P)$-interior $H$-algebras
\[
\theta_0 : \Br_P(e_P\Endom_\ringO(M)e_P) \to \Endom_k(M_0).
\]
\item If $(M'_0,\theta'_0)$ is another such pair, then there exists a linear character $\chi : H/PC_G(P)\to k^\times$ and an isomorphism of $kH\bar e_P$-modules $\phi:\chi_*M_0\to M'_0$ (where $\chi_*M_0$ means the $kH$-module $M_0$ twisted by $\chi$), which induces a commutative diagram
\[
\xymatrix{
 & \Br_P(e_P\Endom_\ringO(M)e_P) \ar[ld]_{\theta_0} \ar[rd]^{\theta'_0} \\
\Endom_k(M_0) \ar[rr]^{u\,\mapsto\, \phi u\phi^{-1}} && \Endom_k(M'_0).
}
\]
\item If $(Q,e_Q,V)$ is a source triple of an indecomposable direct summand of $M_0$, then there is a source triple $(Q',e_Q',V')$ of an indecomposable direct summand of $M$ such that $(P,e_P)\leqslant(Q,e_Q)\leqslant (Q',e_Q')$, and that $V$ is a direct summand of the $P$-slashed module $[\Res^{Q'}_Q V']\sla P$.
\end{enumerate}
\end{lem}

The pair $(M_0,\theta_0)$, or just the $kH\bar e_P$-module $M_0$, is called a $(P,e_P)$-slashed module attached to $M$ over the group $H$. We will usually denote it by $M\sla{P,e_P}$. If $M$ is a $p$-permutation module, then there is a canonical choice of $(P,e_P)$-slashed module attached to $M$: the Brauer quotient $\Br_{(P,e_P)}(M) = \Br_P(e_P M)$, together with the natural isomorphism $\Br_P(e_P\Endom_\ringO(M)e_P) \simeq \Endom_k(\Br_P(e_P M))$.
In general, there is no such canonical choice.

Let $M$ be a Brauer-friendly $\ringO Ge$-module, $(P,e_P)\trianglelefteqslant (Q,e_Q)$ be two $e$-subpairs of $G$, and $H$, $K$ be two subgroups of $G$ such that $PC_G(P)\leqslant H\leqslant N_G(P,e_P)$ and $QC_G(Q)\leqslant K\leqslant N_H(Q,e_Q)$.
Let the pair $(M_0,\theta_0)$ be a $(P,e_P)$-slashed module attached to $M$ over the group $H$, and the pair $(M_1,\theta_1)$ be a $(Q,e_Q)$-slashed module attached to $M_0$ over the group $K$. 
As appears in the proof of \cite[Theorem 19]{Biland2013}, there is a natural isomorphism \linebreak $\psi_1 : \Br_{PQ}(e_{PQ}\Endom_\ringO (M)e_{PQ}) \to \Br_{PQ}(\bar e_{PQ} \Br_Q(e_QMe_Q) \bar e_{PQ})$. 
Set $\theta'_1 = \theta_1\circ \Br_{PQ}(\bar e_{PQ}\theta_0\bar e_{PQ})\circ \psi_1 :\Br_{PQ}(e_{PQ}\Endom_\ringO (M)e_{PQ})\to \Endom_k(M_1)$. The following lemma expresses the transitivity of the slash construction.

\begin{lem}
\label{lem:TranSlash}
With the above notations, the pair $(M_1,\theta'_1)$ is a $(PQ,e_{PQ})$-slashed module attached to $M$ over the group $K$.
\end{lem}

The next lemma will allow us to lift certain indecomposable direct summands through the slash construction.

\begin{lem}
\label{lem:DirSumdLift}
Let $M$ be a Brauer-friendly $\ringO Ge$-module, and $(P,e_P)$ be an $e$-subpair. Let $(M_0,\theta_0)$ be a $(P,e_P)$-slashed module attached to $M$ over the group $N_G(P,e_P)$. If the $kN_G(P,e_P)$-module $M_0$ admits an direct summand with vertex $P$, then the $\ringO G$-module $M$ admits an indecomposable direct summand with vertex subpair $(P,e_P)$.
\end{lem}

\begin{proof} Let $X_0$ be an indecomposable direct summand of $M_0$ with vertex $P$. Then there exists a primitive idempotent $i_0$ of the algebra $\Endom_{kN_G(P,e_P)}(M_0)$ such that $X_0=i_0M_0$, and moreover $i_0$ lies in the ideal $\Tr_P^{\!N_G(P,e_P)}(\Endom_k(M_0))$.

We consider the projection map $\beta:\Endom_{\ringO P}(L)\to \Endom_\ringO(M_0)$, defined by $\beta(u)=\theta_0\circ \br_P(e_Pue_P)$.
For any element $u\in \Endom_{\ringO P}(M)$, we have
\begin{align*}
\beta\circ\Tr_P^G(e_P \, u \, e_P)
\ &= \!\!\!\!
\sum_{g\in N_G(P,e_P)\setminus G/P} 
\!\!\!\!
\theta_0\circ \br_{P} \circ \Tr_{N_{^gP}(P,e_P)}^{N_G(P,e_P)}
(e_P \, \null^g\!e_P \, \null^gu \, \null^ge_P \, e_P)
\\
&=\ 
\Tr_P^{\!N_G(P,e_P)}\circ\, \beta(u)
\end{align*}
This computation proves that the map $\beta$ sends the ideal $\Tr_P^G(\Endom_{\ringO P}(M))$ of the algebra $\Endom_{\ringO G}(M)$ onto the ideal $\Tr_P^{\!N_G(P,e_P)}(\Endom_k(M_0))$ of the algebra $\Endom_{kN_G(Q,e_Q)}(M)$. Thus we know from \cite{BroueRouquier} that the primitive idempotent $i_0$ can be lifted through the map $\beta$, \emph{i.e.}, there exists a primitive idempotent $i$ of the algebra $\Endom_{\ringO G}(M)$ such that $i\in \Tr_P^G(\Endom_{\ringO P}(M))$ and $i_0=\beta(i)$. Then the indecomposable $\ringO G$-module $X=iM$ is a relatively $P$-projective direct summand of $M$. Moreover, $X$ admits $X_0$ as a $(P,e_P)$-slashed module over the group $N_G(P,e_P)$, so $(P,e_P)$ is a vertex subpair of $X$. 
\end{proof}

%=====================
\section{Slashed modules and centrally controlled blocks}
\label{sec:CentControl}

Let $e$ be a block of a finite group $G$. 
We say that a subgroup $H$ of $G$ controls (\emph{resp.} strongly controls) the $e$-fusion in $G$ with respect to a given maximal subpair $(D,e_D)$ if the defect group $D$ is contained in $H$ and, for any non-trivial $e$-subpair $(Q,e_Q)$ contained in $(D,e_D)$,
\[
N_G(Q,e_Q) \leqslant C_G(Q)\, H \qquad\quad \text{(\emph{resp.} } N_G(Q,e_Q) \leqslant O_{p'}(C_G(Q))\, H\, ).
\]

If $H=C_G(P)$ is the centraliser of an $e$-subpair $(P,e_P)$ contained in $(D,e_D)$, then both conditions imply that the Brauer categories $\catbr(G,e)$ and $\catbr(C_G(P),e_P)$ are equivalent and, in particular, that $N_G(P,e_P) = C_G(P)$.

In the next two lemmas, we assume that $e$ is a block of a finite group $G$, and that $(P,e_P)$ be an $e$-subpair of $G$ such that the centraliser $C_G(P)$ controls the $e$-fusion in $G$ with respect to a maximal subpair $(D,e_D)$ that contains $(P,e_P)$. In this context of centrally controlled blocks, the ambiguity of the definition of slashed modules in the previous section can be lifted for well-behaved Brauer-friendly modules.

\begin{lem}
\label{lem:CanSlashMod}
Let $(Q,e_Q)$ be a subpair of $(D,e_D)$, and $H$ be a subgroup of $G$ such that $QC_G(Q)\leqslant H\leqslant N_G(Q,e_Q)$. Denote by $e_{PQ}$ the unique block of the group $C_G(PQ)$ such that $(PQ,e_{PQ})\leqslant (D,e_D)$. Let $M$ be a Brauer-friendly $\ringO Ge$-module.
Assume that the slashed module $M\sla{P,e_P}$ is a $p$-permutation $kC_G(P)\bar e_P$-module, and that a slashed module $M\sla{PQ,e_{PQ}}$ is non-zero. Then there exists a unique isomorphism class of $(Q,e_Q)$-slashed module $M\sla{Q,e_Q}$ over the group $H$ such that
\[
M\sla{Q,e_Q}\sla{PQ,\bar e_{PQ}} \simeq \Br_{(PQ,\bar e_{PQ})} (M\sla{P,e_P}).
\] 
\end{lem}

\begin{proof}
Let the pair $(M_0,\theta_0)$ be any $(Q,e_Q)$-slashed module over the group $H$ attached to $M$. Let the pair $(M_1,\theta_1)$ be a $(P,e_{PQ})$-slashed module attached to $M_0$. This slashed module is defined over the centraliser $C_H(P) = N_H(P,e_{PQ})$, hence it is uniquely defined up to isomorphism. Let the pair $(M_2,\theta_2)$ be a $(P,e_P)$-slashed module attached to $M$. Similarly, this slashed module is uniquely defined over the centraliser $C_G(P)$. Set $M_3= \Br_{(PQ,\bar e_{PQ})}(M_2)$, restricted to a $kC_H(P)$-module, and let $\theta_3 : \Br_{PQ}(\Endom_\ringO(\bar e_{PQ}M_2)) \simeq \Endom_k(M_3)$ be the natural isomorphism.

As in the discussion before Lemma \ref{lem:TranSlash}, define from $\theta_0$ an $\theta_1$ a map $\theta'_1$ such that the pair $(M_1,\theta'_1)$ is a $(PQ,e_{PQ})$-slashed module attached to $M$ over the group $C_H(P)$. Define similarly from $\theta_2$ and $\theta_3$ an isomorphism $\theta'_3$ such that the the pair $(M_3,\theta'_3)$ is a $(PQ,e_{PQ})$-slashed module attached to $M$ over the group $C_H(P)$.

By Lemma \ref{lem:SlashMod} (ii), there exists a linear character $\chi:C_H(P)/C_G(PQ)\to k^\times$ and an isomorphism of slashed modules $(\chi_*M_1, \theta'_1)\simeq (M_3, \theta'_3)$. By control of fusion, the inclusion map $C_H(P)\to H$ induces an isomorphism $C_H(P)/C_G(PQ)\simeq H/C_G(Q)$. Thus twisting the non-zero $kC_H(P)\bar e_{PQ}$-module $M_1$ by a linear character of the group $C_H(P)/C_G(PQ)$ amounts to twisting the $kH\bar e_{Q}$-module $M_0$ by a linear character of the group $H/C_G(Q)$. This proves the existence and uniqueness, up to isomorphism, of a pair $(M_0,\theta_0)$ such that the corresponding $(PQ,e_{PQ})$-slashed module $(M_1, \theta'_1)$ is isomorphic to $(M_3, \theta'_3)$.
\end{proof}

In the next lemma, the normal subgroup $H$ could be the centraliser of a normal $e$-subpair of $G$, or just $G$ itself.

\begin{lem}
\label{lem:UnicM}
Let $V$ be a capped indecomposable endo-permutation $\ringO D$-module such that the class $\Defres^D_{D/P} [V]$ in the Dade group $\dade(D/P)$ is trivial. Assume that the triple $(D,e_D,V)$ is fusion-stable in $G$, and identify $V$ to an $\ringO\Delta D$-module through the diagonal isomorphism.
Let $H$ be a normal subgroup of $G$ such that, for any subpair $(Q,e_Q)\leqslant (D,e_D)$, the idempotent $e_Q$ lies in the subalgebra $kC_H(Q)$ of $kC_G(Q)$. 
Then, up to isomorphism, there is a unique Brauer-friendly $\ringO(H\times C_H(P))\Delta C_G(P)$-module $M$ with source triple $(\Delta D,e_D\otimes e_D^\bop, V)$ such that a slashed module $M\sla{\Delta P,e_P\otimes e_P^\bop}$ admits a direct summand isomorphic to the $k(C_H(P)\times C_H(P))\Delta C_G(P)$-module $kC_H(P)e_P$.
\end{lem}

\begin{proof}
By Lemma \ref{lem:SlashMod} (iii), whenever $M$ is an indecomposable Brauer-friendly $k(H\times C_H(P))\Delta G$-module with source triple $(\Delta D,e_D\otimes e_D^\bop,V)$, the slashed module $M\sla{\Delta P,e_P\otimes e_P^\bop}$ is a $p$-permutation module. Thus we can set
\[
M'\sla{\Delta D,e_D\otimes e_D^\bop} = \Br_{(\Delta D,e_D\otimes e_D^\bop)}
(M'\sla{\Delta P}),
\]
in accordance with Lemma \ref{lem:CanSlashMod}. Once this specific slash construction has been chosen, we know from \cite[Theorem 20]{Biland2013} that the mapping $M\mapsto M\sla{\Delta D,e_D\otimes e_D^\bop}$ induces a one-to-one correspondence between the isomorphism classes of indecomposable Brauer-friendly $k(H\times C_H(P))\Delta G$-modules with source triple $(\Delta D,e_D\otimes e_D^\bop,V)$ and the the isomorphism classes of indecomposable Brauer-friendly $k(C_H(D)\times C_H(D))\Delta N_G(D,e_D)$-modules with source triple $(\Delta D,e_D\otimes e_D^\bop,k)$. The same correspondence, applied to $k(C_H(P)\times C_H(P))\Delta C_G(P)$-modules with source triple $(\Delta D,e_D\otimes e_D^\bop,k)$, implies that the slashed module $M\sla{\Delta P,e_P\otimes e_P^\bop}$ admits $kC_H(P)e_P$ as a direct summand if, and only if, $M\sla{\Delta D,e_D\otimes e_D^\bop } \simeq kC_H(D)e_D$. This proves the lemma.
\end{proof}

%==============
\section{Understanding the local situation}
\label{sec:local case}

In this section, we work directly over the residue field $k$, \emph{i.e.}, we set $\ringO=k$.
We explore an equivariant version of Morita equivalences, the existence of which is proven in \cite{KuelshammerRobinson}. Those are the building blocks that we will glue together to obtain a stable equivalence in the Section \ref{sec:GlobalCase}. 

Let us fix a few notations that will hold throughout the present section. Let $P$ be a $p$-subgroup of a finite group $G$, and $e$ be a block of $G$ such that $\br_P(e)\neq0$. We choose, once and for all, a maximal $e$-subpair $(D,e_D)$ such that $P\leqslant D$. For any subgroup $Q$ of $D$, we denote by $e_Q$ the unique block of $C_G(Q)$ such that $(Q,e_Q)\leqslant (D,e_D)$. Let $H$ be a normal subgroup of $G$ such that, for any subgroup $Q$ of $D$, the block $e_Q$ of the group algebra $kC_G(Q)$ lies in the subalgebra $kC_H(Q)$ (for instance, $H$ may be the centraliser of a normal $e$-subpair of $G$). Assume that
\[
G \ = \ O_{p'}(H) \ C_G(P).
\]
By elementary group theory, this factorisation implies that $P$ is an abelian $p$-group, and that the centraliser $C_G(P)$ controls the $p$-fusion in the group $G$. 
 In particular, we have $N_G(D)\leqslant C_G(P)$. By Brauer's first main theorem, it follows that the idempotent $e_P=\br_P(e)$ is a block of the group $C_G(P)$. Thus $(P,e_P)\leqslant (D,e_D)$ is an $e$-subpair of the group $G$, and the centraliser $C_G(P)$ controls the $e$-fusion with respect to the maximal subpair $(D,e_D)$.
The main result of this section is the following.

\begin{thm}
\label{thm:local equivariant Morita equivalence}
With the above notations,
\begin{enumerate}[(i)]
\item
Let $S$ be any $p'$-subgroup of $H$ such that $S\trianglelefteqslant G$ and $G=SC_G(P)$; let $b$ be any $D$-stable block of the group $S$ such that $e_D\br_D(b)\neq 0$. The $\Delta D$-algebra $kSb\otimes kC_S(P)\br_P(b)^\bop$ defines a class $v$ in the Dade group $\dade(\Delta D)$ that is independent of the choice of $S$ and~$b$.
\item
Let $V$ be a capped indecomposable endo-permutation $k\Delta D$-module that belongs to the class $v$. The source triple $(\Delta D,e_D\otimes e_D^\bop,V)$ is fusion-stable in the group $(H\times C_H(P))\Delta C_G(P)$.
\item There exists a unique indecomposable 
$k(H\times C_H(P))\Delta C_G(P)$-module $M$ with source triple $(\Delta D,e_D\otimes e_D^\bop,V)$ such that the slashed module $M\sla{\Delta P}$ is isomorphic to the $k(C_H(P)\times C_H(P))\Delta C_G(P)$-module $kC_H(P)e_P$.
\item
The module $M$ induces a $G/H$-equivariant Morita equivalence
$kHe\ \sim \ kC_H(P)e_P$.
\end{enumerate}
\end{thm}

We reach the proof of this theorem through a series of lemmas. We choose, once and for all, a normal $p'$-subgroup $S$ of $H$ such that $S\triangleleft G$ and $G=SC_G(P)$. For instance, we could choose $S=O_{p'}(H)$. Let $b_D$ be a block of the algebra $kC_S(D)$ such that $e_Db_D\neq 0$, \emph{i.e.}, the block $e_D$ of $C_G(D)$ covers the block $b_D$ of the normal subgroup $C_S(D)$ (as defined in \cite{Dade1973}). 

We may consider $(D,b_D)$ as a maximal subpair of the nilpotent group $SD$. Let $b$ be the block of $SD$ such that $(D,e_D)$ is a maximal $b$-subpair. In other words, $b$ is a $D$-stable block of the $p'$-group $S$, and $b_D = \br_D(b)$. Similarly, for any subgroup $Q$ of $D$, the Brauer map $\br_Q$ defines a one-to-one correspondence between the set of $Q$-stable blocks of $S$ and the set of all blocks of $C_S(Q)$, by Brauer's first main theorem. In particular, the idempotent $b_Q=\br_Q(b)$ is a block of the group $C_S(Q)$.

\begin{lem}
\label{lem:choice of the block b}
For any subpair $(Q,e_Q)\leqslant (D,e_D)$, the block $e_Q$ of the group $C_G(Q)$ covers the block $b_Q=\br_Q(b)$ of the normal subgroup $C_S(Q)$. In particular, the block $e$ covers the block $b$.
\end{lem}

\begin{proof}
By construction, the block $e_D$ covers the block $b_D$. We use descending induction to generalise this to any subpair of $(D,e_D)$. Let $(Q,e_Q)$ be a proper subpair of $(D,e_D)$. We assume that, for any subpair $(R,e_R)$ with $(Q,e_Q)<(R,e_R)\leqslant (D,e_D)$, the block $e_R$ of the group $C_G(R)$ covers the block $b_R$ of the normal subgroup $C_S(R)$. Then we let $(R,e_R)$ be the normaliser subpair of $Q$ in $(D,e_D)$, which strictly contains $(Q,e_Q)$. Thus we have $e_R\br_R(e_Q) = e_R$ and, by induction, $e_Rb_R\neq 0$. These imply $\br_R(e_Q)b_R\neq 0$, hence $\br_R(e_Qb_Q)=\br_R(e_Q)b_R\neq 0$. Thus we obtain $e_Qb_Q\neq0$ and the block $e_Q$ covers the block $b_Q$. This completes the induction step.
\end{proof}

We denote by $G_b$ and $H_b$ the stabilisers of the block $b$ in the groups $G$ and $H$ respectively, and by $b'=\Tr_{G_b}^G(b)$ the sum of all $G$-conjugates of $b$. Then $eb'=e$ and $e\in kHb'$; moreover $e_b=eb$ is a block of the algebra $kG_b$, it lies in the subalgebra $kH_b$, and the $(kHe,kH_be_b)$-bimodule $kHe_b$ induces a ($G_b/H_b$-equivariant) Morita equivalence $kHe\sim kH_be_b$, as is proven for example in \cite{Harris2007}.

\begin{lem}
\label{lem:block b fusion control}
The subgroup $C_{G_b}(P)$ controls the $e$-fusion in $G$ with respect to the maximal subpair $(D,e_D)$.
\end{lem}

\begin{proof}
Let $(Q,e_Q)$ be a subpair of $(D,e_D)$ and let $g\in G$ be such that  $^g(Q,e_Q)\leqslant (D,e_D)$. The centraliser $C_G(P)$ controls the $p$-fusion in $G$, so we may suppose $g\in C_G(P)$. Then we obtain $^g(PQ,e_{PQ}) \leqslant (D,e_D)$, so we may suppose $P\leqslant Q$. 
The inclusion $^g(Q,e_Q)\leqslant (D,e_D)$ implies $^g e_Q=e_{^gQ}$. So the block $^g e_Q$ of $C_G({^gQ})$ covers the block $b_{^gQ}$ of $C_S({^gQ})$, and the block $e_Q$ of $C_G(Q)$ covers the block $^{g^{-1}}b_{^gQ}$ of $C_S(Q)$. As the blocks $b_Q$ and $^{g^{-1}}b_{^gQ}$ of $C_S(Q)$ are covered by the same block $e_Q$ of $C_G(Q)$, they must be conjugate in $C_G(Q)$: there exists $k\in C_G(Q)$ such that $b_Q=\,^{kg^{-1}}b_{^gQ}$. Then we get
\[
\br_{Q}(b)
\ = \ 
b_Q
\ = \ 
\,^{kg^{-1}}\!b_{^gQ}
\ = \ 
\,^{kg^{-1}}\!\br_{^gQ}(b)
\ = \ 
\br_{Q}(\,^{kg^{-1}} \!b).
\]
As we have already mentionned, the correspondence $b\leftrightarrow \br_Q(b)$ is one-to-one, so we obtain $b=\,^{kg^{-1}}\!b$. Hence the element $h=gk^{-1}$ lies in $G_b$. Notice that $g$ and $k$ both centralise the $p$-group $P$ by assumption, so we have $h\in C_{G_b}(P)$, and $g\in C_G(Q)\,C_{G_b}(P)$. Thus $N_G(Q,e_Q)\leqslant C_G(Q) \,C_{G_b}(P)$.
\end{proof}

For any element $x\in P$, we denote by $K_x\in kG$ the class sum of $x$, \emph{i.e.}, $K_x = \Tr_{C_G(x)}^G (x)$ (see \cite{AlperinBroue} for a definition of the relative trace map). We have supposed $G=SC_G(P)$ hence $G=SC_G(x)$. So, for any subgroup $T$ of $G$ that contains $S$, the natural map $T/C_T(x)\to G/C_G(x)$ is a bijection and $K_x=\Tr_{C_T(x)}^T(x)$. We have in particular $K_x=\Tr_{C_S(x)}^S(x)$. Since the subgroup $S$ is normal in $G$, it follows that the class sum $K_x$ lies in the subset $kSx=xkS$ of the algebra $kG$, and that the element $xK_{x^{-1}}$ lies in $kS$.

In \cite{Robinson1986} and \cite{Robinson2009}, Robinson makes great use of the central unit $K_xe \in kGe$ to deal, respectively, with the situation $G = O_{p'}(G) \, C_G(P)$ and with a minimal counter-example to the odd $Z^*_p$-theorem. The following lemmas highlight once again the importance of the class sum $K_x$. In order to deal efficiently with it, we need more notations.

We will consider the group $G$ as a subgroup of the direct product $G\times P$, via the embedding $g\mapsto (g,1)$. Since the $p$-group $P$ is abelian, we can consider the $p$-subgroup $P_1 = \{(x,x^{-1})\,;\ x\in P\}$ of $G\times P$. For an element $x\in P$, we will usually write $x_1=(x,x^{-1})\in P_1$; conversely, for an element $x_1\in P_1$, we will write $x$ for the unique element of $P$ such that $x_1=(x,x^{-1})$. We will see the group $G\times P$ as the semi-direct product $GP_1$. Notice that any subgroup of $G$ that is normalised by $P$ is also normalised by $P_1$. Thus we can consider the subgroups $HP_1$, $SP_1$, $G_bP_1$, \emph{etc.} The group $P_1$ centralises the defect group $D$ and the blocks $e$, $e_D$, $b$, \emph{etc.}

\begin{lem}
\label{lem:mysterious interior structure}
\begin{enumerate}[(i)]
\item 
The maps $\iota_b:SP_1\to (kSb)^\times$ and $\gamma_b:G_bP_1\to \Autom_\catalg(kSb)$ defined by

\smallskip

\centerline{$
\iota_b(sx_1) = sxK_{x^{-1}}b
\qquad ;\qquad 
\gamma_b(gx_1)(a)= (gx)^{-1}a(gx)
$}

for $s\in S$, $g\in G_b$, $x_1\in P_1$, $a\in kSb$, 
make $kSb$ an $SP_1$-interior $G_bP_1$-algebra.
\item 
The maps $\iota_{b'}:HP_1\to (kHb')^\times$ and $\gamma_{b'}:GP_1\to \Autom_\catalg(kHb')$ defined by

\smallskip

\centerline{$
\iota_{b'}(hx_1) = hxK_{x^{-1}}b'
\qquad ;\qquad 
\gamma_{b'}(gx_1)(a)= (gx)^{-1}a(gx)
$}

for $h\in H$, $g\in G$, $x_1\in P_1$, $a\in kHb'$, 
make $kHb'$ an $HP_1$-interior $GP_1$-algebra.
\item 
The maps $\iota_e:HP_1\to (kHe)^\times$ and $\gamma_e:GP_1\to \Autom_\catalg(kHe)$ defined by

\smallskip

\centerline{$
\iota_e(hx_1) = hxK_{x^{-1}}e
\qquad ;\qquad 
\gamma_e(gx_1)(a)= (gx)^{-1}a(gx)
$}

for $h\in H$, $g\in G$, $x_1\in P_1$, $a\in kHe$, 
make $kHe$ an $HP_1$-interior $GP_1$-algebra. 
\end{enumerate}
\end{lem}

\begin{proof}
We consider the idempotents $b$ and $b_P=\br_P(b)$ as respective blocks of the $p$-nilpotent groups $SP$ and $C_S(P)P$, both with defect group $P$. It is well known that the block algebras $kSPb$ and $kC_S(P)Pb_P$ are both Morita equivalent to $kP$, so the centers $Z(kSPb)$ and $Z(kC_S(P)Pb_P)$ are both isomorphic to $Z(kP)=kP$; in particular, they have the same dimension. The Brauer map $\br_P$ induces an algebra morphism $\beta:Z(kSPb)\to Z(kC_S(P)Pb_P)$. We have $kPb_P\subseteq Z(kC_S(P)Pb_P)$. Moreover the natural map $kC_S(P)\otimes kP\to kC_S(P)P$ is an isomorphism, so $\dim_k(kPb_P)= |P|=\dim_k Z(kC_S(P)Pb_P)$. Hence $Z(kC_S(P)Pb_P)=kPb_P$.

Let $x$ in $P$ be fixed. Since $C_G(P)$ controls the $p$-fusion, no proper conjugate of $x$ lies in $C_G(P)$. So $\beta(K_xb) = \br_P(K_x)\br_P(b) = xb_P$, which proves that the morphism $\beta$ is onto. Since its domain and codomain have the same dimension over $k$, $\beta$ is an isomorphism. Thus the element $K_x b=\beta^{-1}(xb_P)$ is invertible in $Z(kSPb)$ and the map $x\mapsto K_x b$ is a group morphism $P\to Z(kSPb)^\times$. Moreover the group $P$ is abelian, so the map $\iota_b:SP_1\to (kSb)^\times$ of (i) is indeed well-defined and a group morphism. The rest of the statement in (i) is straightforward.

Furthermore, the algebra $kSb'$ is the direct product of the $kSc$ where $c$ runs over the set of $G$-conjugates of $b$, and $xK_{x^{-1}}b' = \sum_{c} xK_{x^{-1}}{c}$ for any $x\in P$. So $xK_{x^{-1}}b'$ is invertible in $kSb'$ and the map $x_1\mapsto xK_{x^{-1}} b'$ is a group morphism $P_1\to (kSb')^\times$, which extends to the group morphism $\iota_{b'}:HP_1\to (kHb')^\times$ of (ii). Notice that cutting off the central idempotent $e$ cannot harm, so (iii) follows immediately.
\end{proof}

\begin{lem}
\label{lem:IsoIndAlg}
There is a natural isomorphism of $HP_1$-interior $GP_1$-algebras
\[
\phi:kHb' \ \longisom \ \Ind_{(SP_1\times SP_1)\Delta G_b}^{(HP_1\times HP_1)\Delta G} kSb.
\]
\end{lem}

\begin{proof}
Let us write $A=\Ind_{(SP_1\times SP_1)\Delta G_b}^{(HP_1\times HP_1)\Delta G} kSb$.
On the one hand, we have from Section \ref{sec:GenDef} an isomorphism of $H$-interior $G$-algebras $kHb'\to \Ind_{(S\times S)\Delta G_b}^{(H\times H)\Delta G} kSb$.
On the other hand,
the natural map $G/S\to GP_1/SP_1$ is bijective so the natural map 
$
\Ind_{(S\times S)\Delta G_b}^{(H\times H)\Delta G} kSb \to \Ind_{(SP_1\times SP_1)\Delta G_b}^{(HP_1\times HP_1)\Delta G} kSb
$
is an isomorphism of $H$-interior $G$-algebras. By composition, we obtain the map $\phi:kHb'\to A$, which appears to be an isomorphism of $H$-interior $G$-algebras. By definition, we have $\iota_{b'}(x_1) = xK_{x^{-1}} b'$, where the element $xK_{x^{-1}}$ lies in $S$. Since $\phi$ is an isomorphism of left $kS$-modules, we obtain $\phi(\iota_{b'}(x_1)) = xK_{x^{-1}}\cdot 1_A$. Then it follows from the definition of induced interior algebras that $\iota_A(x_1) = xK_{x^{-1}}\cdot 1_A$. Thus $\phi$ is also an isomorphism of $P_1$-interior algebras, and the lemma is proven.
\end{proof}

We now consider the $HP_1$-interior $GP_1$-algebra $kHe$ as a $k(HP_1\times HP_1)\Delta G$-module.

\begin{lem}
\label{lem:source of kGe}
The indecomposable $k(HP_1\times HP_1)\Delta G$-module $kHe$ is Brauer-friendly with source triple $((P_1\times P_1)\Delta D, e_D\otimes e_D^\bop, W)$, where
the source $W$ is any capped indecomposable direct summand of the restriction $\Res^{(SP_1\times SP_1)\Delta G_b}_{(P_1\times P_1)\Delta D} kSb$. 
\end{lem}

\begin{proof}
Let us write $K=(SP_1\times SP_1)\Delta G_b$.
The field $k$ is algebraically closed and $S$ is a $p'$-group, so the block algebra $kSb$ is a matrix algebra. It follows that the structure map of the $(kSb,kSb)$-bimodule $kSb$ is an isomorphism of $K$-algebras 
$
kSb\otimes kSb^\bop \simeq \Endom_k(kSb).
$
In particular, this proves that $kSb$ is an endo-$p$-permutation $kK$-module. 

For any  element $(g,h)$ of the group $(H_bP_1\times H_bP_1)\Delta G_b$, let $R$ be a Sylow $p$-subgroup of the intersection $K\cap\null^{(g,h)}K$. The $k(S\times S)R$-module $\Res^{\, K}_{(S\times S)R} kSb$ is simple and belongs to the block $b\otimes b^\bop$ of the $p$-nilpotent group $(S\times S)R$. Since the pair $(g,h)$ stabilises the block $b\otimes b^\bop$, the $k(S\times S)R$-module $\Res^{^{(g,h)}\!K}_{(S\times S)R} g(kSb)h^{-1}$ is still simple and belongs to the same block $b\otimes b^\bop$. Since a block of a $p$-nilpotent group contains only one isomorphism class of simple modules, there must be an isomorphism of $k(S\times S)R$-modules
\[
\Res^{\,K}_{(S\times S)R} kSb
\ \simeq \ 
\Res^{^{(g,h)}\!K}_{(S\times S)R} g(kSb)h^{-1}.
\]
It follows that the restrictions  $\Res^{\,K}_{K\cap\null ^{(g,h)}\!K} kSb$ and $\Res^{^{(g,h)}\!K}_{K\cap\null ^{(g,h)}\!K} g(kSb)h^{-1}$ are compatible endo-$p$-permutation $k(K\cap\null ^{(g,h)}\!K)$-modules. By Urfer's criterion \cite[Lemma 1.3]{Urfer2007} for the induction of endo-$p$-permutation modules, we deduce that the $k(H_bP_1\times H_bP_1)\Delta G_b$-module
\[
kH_bb \ \simeq \ \Ind_{(SP_1\times SP_1)\Delta G_b}^{(H_bP_1\times H_bP_1)\Delta G_b} kSb
\]
is an endo-$p$-permutation module. Then its direct summand $kH_be_b$ is also an endo-$p$-permutation $k(H_bP_1\times H_bP_1)\Delta G_b$-module. We now determine a vertex subpair of this indecomposable module. The commutation of induction and the Brauer functor brings an isomorphism of $(C_{H_b}(P)P_1\times C_{H_b}(P)P_1)\Delta C_{G_b}(P)$-interior algebras
\[
\Br_{P_1\times P_1}(\Endom_k(kH_bb))
 \ \to \ \Ind_{[(C_S(P)P_1\times C_S(P)P_1)\Delta C_{G_b}(P)]^2}^{[(C_{H_b}(P)P_1\times C_{H_b}(P)P_1)\Delta C_{G_b}(P)]^2} \Br_{P_1\times P_1}(\Endom_k(kSb)).
\]
The natural isomorphism of 
$(C_S(P)P_1\times C_S(P)P_1)\Delta C_{G_b}(P)$-interior algebras
$\Br_{P_1\times P_1}(\Endom_k(kSb)) \simeq \Endom_k(kC_S(P)b_P)$ and the isomorphism of Lemma \ref{lem:IsoIndAlg} then bring an isomorphism of 
$(C_{H_b}(P)P_1\times C_{H_b}(P)P_1)\Delta C_{G_b}(P)$-interior algebras
\[
\Br_{P_1\times P_1}(\Endom_k(kH_bb))
\ \simeq\
\Endom_k(kC_{H_b}(P)b_P).
\]
It follows that the slashed module $kH_be_b\sla{P_1\times P_1}$ is isomorphic to $kC_{H_b}(P)\br_P(e_b)$. So a vertex of the indecomposable $k(H_bP_1\times H_bP_1)\Delta G_b$-module $kH_be_b$ contains the $p$-group $P_1\times P_1$. Since $kC_{H_b}(P)\br_P(e_b)$ is a $p$-permutation $k(C_{H_b}(P)\times C_{H_b}(P))\Delta C_{G_b}(P)$-module, the slash construction may coincide with the Brauer functor from this point on. The images of the block algebra $kC_{H_b}(P)\br_P(e_b)$ by Brauer functors are well-known, so we can use the transitivity of the slash construction for endo-$p$-permutation modules, and conclude that a vertex subpair of $kH_b e_b$ is $((P_1\times P_1)\Delta D, e_Db_D\otimes e_D^\bop b_D^\bop)$.

Since the indecomposable $k(H_bP_1\times H_bP_1)\Delta G_b$-module $kH_be_b$ is a direct summand of the endo-$p$-permutation module $\Ind_{(SP_1\times SP_1)\Delta G_b}^{(H_bP_1\times H_bP_1)\Delta G_b} kSb$, a source $W$ of $kH_be_b$ with respect to the above vertex subpair is isomorphic to any capped indecomposable direct summand of the restriction $\Res^{(SP_1\times SP_1)\Delta G_b}_{(P_1\times P_1)\Delta D} kSb$.
As a consequence of the vertex-preserving Morita equivalence of \cite[Theorem 1.6]{Harris2007}, the induced module
\[
kHe \ \simeq \ 
\Ind_{(H_bP_1\times H_bP_1)\Delta G_b}^{(HP_1\times HP_1)\Delta G}\ kH_b e_b
\]
is indecomposable and admits the source triple $((P_1\times P_1)\Delta D, e_D\otimes e_D^\bop,W)$. Moreover, $kH_be_b$ is an endo-$p$-permutation module, so the endo-permutation source pair $((P_1\times P_1)\Delta D,W)$ is fusion-stable in the group $(H_bP_1\times H_bP_1)\Delta G_b$. By Lemma \ref{lem:block b fusion control}, the subgroup $(H_bP_1\times H_bP_1)\Delta G_b$ controls the $e\otimes e^\bop$-fusion in the group $(HP_1\times HP_1)\Delta G$. Thus the source triple $((P_1\times P_1)\Delta D, e_D\otimes e_D^\bop,W)$ is fusion-stable in the group $(HP_1\times HP_1)\Delta G$, and the idecomposable module $kHe$ is Brauer-friendly. 
\end{proof}

Let $M=kHe\sla{1\times P_1,e\otimes e_P^\bop}$ be a slashed module attached to the Brauer-friendly $k(HP_1\times HP_1)\Delta G$-module $kHe$.  Remember that $N_G(P,e_P) = C_G(P)$, so the slash construction is unambiguous as long as only the $p$-groups $P$ and $P_1$ are concerned. Since $e_P=\br_P(e)$, we may also omit the blocks in subpairs concerned only with $P$ and $P_1$. For instance, we may write $M= kHe\sla{1\times P_1}$. From now on, we will consider $M$ as a $k(H\times C_H(P))\Delta C_G(P)$-module, thus forgetting the remaining left action of $P_1$.

\begin{lem}
\label{lem:LocMoriProof}
The $k(H\times C_H(P))\Delta C_G(P)$-module $M$ induces a $G/H$-equivariant Morita equivalence $kHe \sim kC_H(P)e_P$.
\end{lem}

\begin{proof}
We apply the slash construction to the $k(SP_1\times SP_1)\Delta G_b$-module $kSb$ and the $k(HP_1\times HP_1)\Delta G$-module $kHb'$ to define a $k(S\times C_S(P))\Delta G_b$-module $L=kSb\sla{1\times P_1}$ and a $k(H\times C_H(P))\Delta C_G(P)$-module $L'=kHb'\sla{1\times P_1}$. We know from Lemma \ref{lem:IsoIndAlg} that there is an isomorphism of $k(HP_1\times HP_1)\Delta G$-modules $kHb'\simeq \Ind_{(SP_1\times SP_1)\Delta G_b}^{(HP_1\times HP_1)\Delta G} kSb$. Moreover $G=SC_G(P)$, so the commutation of induction and the Brauer functor brings an isomorphism of $(HP_1\times C_H(P)P_1)\Delta C_G(P)$-interior algebras
\begin{gather*}
\Ind_{[(SP_1\times C_S(P)P_1)\Delta C_{G_b}(P)]^2}^{[(HP_1\times C_H(P)P_1)\Delta C_G(P)]^2} 
\circ \Br_{1\times P_1} \Endom_k(kSb)
\hspace{5cm} \\ \hspace{5cm}
\to
\Br_{1\times P_1} \circ
\Ind_{[(SP_1\times SP_1)\Delta G_b]^2}^{[(HP_1\times HP_1)\Delta G]^2} 
\Endom_k(kSb).
\end{gather*}
Notice that the $p$-subgroup $P_1$ can be omitted from the induction functors without changing the result. Thus we have an isomorphism of $k(H\times C_H(P))\Delta C_G(P)$-modules 
\[
\Ind_{(S\times C_S(P))\Delta C_{G_b}(P)}^{(H\times C_H(P))\Delta C_G(P)} L
\ \simeq\ 
L'.
\]
Then we look closer at the definition of $L$. Since $kSb$ is a matrix algebra, the structure map of the $(kSb,kSb)$-bimodule $kSb$ is an isomorphism of $(S\times S)$-interior algebras $kSb\otimes (kSb)^\op \to \Endom_k(kSb)$. Applying the Brauer functor $\Br_{1\times P_1}$ turns this into an isomorphism of $(S\times C_S(P))$-interior algebras $kSb\otimes (kC_S(P)b_P)^\op \to \Endom_k(L)$. So we have an isomorphism of $k(S\times C_S(P))$-modules $L\simeq X\otimes Y^*$, where $X$ is a simple module for the matrix algebra $kSb$ and $Y^*$ is the $k$-dual of a simple module for the matrix algebra $kC_S(P)b_P$. 

We deduce that the $k(S\times C_S(P))$-module $L$ induces a Morita equivalence $kSb\sim kC_S(P)b_P$. By \cite[Theorem 3.4]{Marcus1996} and \cite[Theorem 1.6]{Harris2007}, it follows that the induced module $L'$ induces a Morita equivalence $kHb'\sim kC_H(P)b'_P$. Then the non-zero direct summand $eL'e_P$ induces a Morita equivalence $kHe\sim kC_H(P)e_P$.
\end{proof}

\begin{lem}
\label{lem:source of the local M}
The indecomposable $k(H\times C_H(P))\Delta C_G(P)$-module $M$ is Brauer-friendly with source triple $(\Delta D,e_D\otimes e_D^\bop, V)$, where the endo-permutation $k\Delta D$-module $V$ belongs to the class of the Dade group $\dade(\Delta D)$ that is  defined by the Dade $\Delta D$-algebra $kSb\otimes kC_S(P)b_P^\bop$.
\end{lem}

\begin{proof}
The indecomposable $k(HP_1\times HP_1)\Delta G$-module $kHe$ is Brauer-friendly with source triple $((P_1\times P_1)\Delta D,e_D\otimes e_D^\bop,W)$. We know from Lemma \ref{lem:SlashMod} that the slashed module $M$ is Brauer-friendly. For any subpair $(R,f)$ of the maximal $e\otimes e_P^\bop$-subpair $(D\times D,e_D\otimes e_D^\bop)$ in the group $(H\times C_H(P))\Delta C_G(P)$, the transitivity of the slash construction shows that an $(R,f)$-slashed module $M\sla{R, f}$ attached to $M$ is also a $(1\times P_1)R, \br_{1\times P_1}(f))$-slashed module attached to $kHe$. It follows that $M\sla{R,f}$ is non-zero if, and only if, the subpair $(R,f)$ is contained in $(\Delta D,e_D\otimes e_D^\bop)$ up to conjugation. Thus $(\Delta D,e_D\otimes e_D^\bop)$ is a vertex subpair of $M$.

Let $V$ be the source of $M$ with respect to the above vertex subpair. By Lemma \ref{lem:SlashMod} (iii), the endo-permutation $k\Delta D$-module $V$ is compatible with the slashed module $\Res^{(P_1\times P_1)\Delta D}_{\Delta D}W\sla{1\times P_1}$.
Moreover, we know from Lemma \ref{lem:source of kGe} that $W$ is a capped indecomposable direct summand of the $k\Delta D$-module $kSb$. We have $\Endom_k(kSb)\simeq kSb\otimes kSb^\bop$, so $\Br_{1\times P_1}(\Endom_k(kSb)) \simeq kSb\otimes kC_S(P)b_P^\bop$. Thus the $k\Delta D$-module $V$ is isomorphic to a direct summand of a simple module for the matrix algebra $kSb\otimes kC_S(P)b_P^\bop$.
\end{proof}

Finally, the uniqueness statement of Theorem \ref{thm:local equivariant Morita equivalence} (iii) follows from Lemma \ref{lem:UnicM}.

%=================
\section{Gluing sources}
\label{sec:SourceGlue}

In this section, we work over the local ring $\ringO$. Let $e$ be a block of a finite group $G$, and $(P,e_P)$ be an $e$-subpair of $G$.  We choose, once and for all, a maximal $e$-subpair $(D,e_D)$ that contains $(P,e_P)$, and we assume that the centraliser $C_G(P)$ strongly controls the $e$-fusion in $G$ with respect to the maximal subpair $(D,e_D)$. For any subgroup $Q$ of $D$, we denote by $e_Q$ the unique block of the centraliser $C_G(Q)$ such that $(Q,e_Q)\leqslant (D,e_D)$. 

Let $Q\neq 1$ be a non-trivial subgroup of $D$. Then the $e$-subpair $(N_D(Q),e_{N_D(Q)})$ may be seen as an $e_Q$-subpair of the group $N_G(Q,e_Q)$, although it needs not be maximal. By assumption, we have $N_G(Q,e_Q)\leqslant O_{p'}(C_G(Q)) C_G(P)$. Let $S_Q$ be any normal $p'$-subgroup of $N_G(Q,e_Q)$ such that $S_Q\leqslant C_G(Q)$ and  $N_G(Q,e_Q) \leqslant S_Q \, C_G(P)$. Let $\bar b_Q$ be an $N_D(Q)$-stable block of the group $S_Q$ such that the block $\bar e_{N_D(Q)}$ of $C_G(N_D(Q))$ covers the block $\br_{N_D(Q)}(\bar b_Q)$ of $C_S(N_D(Q))$. The diagonal conjugation action of the group $N_D(Q)$ on the matrix algebra $kS_Q\bar b_Q\otimes kC_{S_Q}(P)\br_P(\bar b_Q)^\bop$ makes it a Dade $N_D(Q)/Q$-algebra, since the normal subgroup $Q$ acts trivially. Let $v_Q$ be the corresponding class in the Dade group $\dade(N_D(Q)/Q)$.

\begin{lem}
\label{lem:compatible sources}
With the notations of \cite{BoucThevenaz},
\begin{enumerate}[(i)]
\item The class $v_Q\in\dade(N_D(Q))$ is independent of the choice of $S_Q$ and $b_Q$. 
\item If $Q\triangleleft R$ are non-trivial subgroups of $D$ , then 
\[
\Defres^{N_D(Q)/Q}_{N_D(Q,R)/R} v_Q = \Res^{N_D(R)/R}_{N_D(Q,R)/R} v_R.
\]
\item If $g\in G$ is such that $\,^g(Q,e_Q)\leqslant (D,e_D)$, then 
\[
\Res^{N_D(Q)}_{N_{D\cap D^g}(Q)} v_Q = \Res^{N_{D^g}(Q)}_{N_{D\cap D^g}(Q)} g^{-1}\cdot v_{^{\hspace{.1mm}g\!}Q}.
\]
\end{enumerate}
\end{lem}

\begin{proof}
Let us fix a non-trivial $p$-subgroup $Q$ of the defect group $D$. We write $G_Q=N_G(Q,e_Q)$ and $H_Q=C_G(Q)$. By assumption, we have  the factorisation $G_Q = S_QC_{G_Q}(P)$, so that all the assumptions of Section \ref{sec:local case} are satisfied. We denote by $M_Q$ the indecomposable Brauer-friendly $k(H_Q\times C_{H_Q}(P))\Delta C_{G_Q}(P)$-module of Theorem \ref{thm:local equivariant Morita equivalence}. Let $V_Q$ be a capped indecomposable direct summand of the $k\Delta N_D(Q)$-module $e_{N_D(Q)}M_Qe_{N_D(Q)}$, which we identify to a $kN_D(Q)$-module through the diagonal isomorphism. We know from Theorem \ref{thm:local equivariant Morita equivalence} that $V_Q$ exists and is an endo-permutation $kN_D(Q)$-module that belongs to the class $\Inf_{N_D(Q)/Q}^{N_D(Q)} v_Q$, and we know from Lemma \ref{lem:BFModSubpair} that the isomorphism class of $V_Q$ depends only on the Brauer-friendly-module $M_Q$ and the subpair $(Q,e_Q)$. This proves (i).

We now take $1\neq Q\triangleleft R\leqslant D$.
On the one hand, $S_Q$ is a normal $p'$-subgroup of $G_Q$ such that $S_Q\leqslant H_Q$ and $G_Q= S_Q H_Q$, and $\bar b_Q$ is an $N_D(Q)$-stable block of $kS_Q$ such that the block $e_{N_D(Q)}$ covers $\br_{N_D(Q)}(\bar b_Q)$. 
We set $G_{Q,R}=N_G(Q,R,e_R)$, $S_{Q,R} = C_{S_Q}(R)$ and $b_{Q,R}=\br_R(b_Q)$. 
Then $S_{Q,R}$ is a normal $p'$-subgroup of $G_{Q,R}$ such that $S_{Q,R}\leqslant H_R$ and $G_{Q,R}= S_{Q,R} C_{G_{Q,R}}(P)$, and $b_{Q,R}$ is an $N_D(Q,R)$-stable block of $S_{Q,R}$ such that the block $e_{N_D(Q,R)}$ covers $\br_{N_D(Q,R)}(b_{Q,R})$. 
Let $v_{Q,R}\in\dade(N_D(Q,R)/R)$ be the class defined by the Dade $N_D(Q,R)$-algebra $kS_{Q,R}\bar b_{Q,R}\otimes kC_{S_{Q,R}}(P)\br_P(\bar b_{Q,R})^\bop$, \emph{i.e.}, $v_{Q,R} = \Defres^{N_D(Q)/Q}_{N_D(Q,R)/R} v_Q$.

On the other hand, let $S_R$ be a normal $p'$-subgroup of $G_R$ such that $S_{R}\leqslant H_R$ and $G_R= S_{R} C_{G_R}(P)$, and $\bar b_R$ be an $N_D(R)$-stable block of $S_{R}$ such that the block $\bar e_{N_D(R)}$ covers $\br_{N_D(R)}(\bar b_{R})$. Then $S_R$ is also a normal $p'$-subgroup of $G_{Q,R}$ such that $S_{R}\leqslant C_G(R)$ and $G_{Q,R}= S_{R} C_{G_{Q,R}}(P)$, and $b_{R}$ is an $N_D(Q,R)$-stable block of $S_R$ such that the block $\bar e_{N_D(Q,R)}$ covers $\br_{N_D(Q,R)}(\bar b_{R})$. Since the class $v_{Q,R}$ is independent of the choice of the subgroup $S_{Q,R}$ and of the block $b_{Q,R}$, it follows that $v_{Q,R}=\Res^{N_D(R)/R}_{N_D(Q,R)/R} v_R$, and (ii) is proven. The proof of (iii) is essentially the same.
\end{proof}

The rest of this article depends on the following assumption.

\begin{hyp}
\label{hyp:gluing sources}
There exists a capped indecomposable endo-permutation $\ringO D$-module $V$ such that the triple $(D,e_D,V)$ is fusion-stable in $G$ and that, for any non-trivial subgroup $Q$ of $D$, 
\[
\Defres^{\,D}_{N_D(Q)/Q} [k\otimes_\ringO V] = v_Q.
\]
\end{hyp}

\begin{lem}
If $e$ is the principal block of the group $G$, or if the defect group $D$ is abelian, or if the prime $p$ is odd and the poset $\mathcal A_{\geqslant 2}(D)$ of elementary abelian subgroups of $D$ of rank at least 2 is connected, then Assumption \ref{hyp:gluing sources} is satisfied.
\end{lem}

\begin{proof}
Firstly, we suppose that $e$ is the principal block of the group $G$. For any non-trivial subgroup $Q$ of the defect group $D$, the principal block $e_Q$ of the group $C_G(Q)$ covers the principal block $b_Q$ of the $p'$-group $S_Q$, so $v_Q$ is the trivial class in the Dade group $\dade(N_D(Q)/Q)$. Thus we can choose $V$ to be the trivial $\ringO D$-module.

Secondly, we suppose that the defect group $D$ is abelian. Then we have $N_D(Q)=D$ for any subgroup $Q$ of $D$. Following \cite{Puig1991}, we consider the function $\mu$ on the set of non-trivial subgroups of $D$ such that $\sum_{1\neq R\leqslant Q} \mu(R) = 1$ for any non-trivial subgroup $Q$ of $D$. We consider the class
\[
v \ = \ 
\sum_{1\neq Q\leqslant D} \mu(Q) \Inf_{D/Q}^{\,D} v_Q
\ \in\ \dade(D).
\]
By Lemma \ref{lem:compatible sources}, the family $(v_Q)_{1\neq Q\leqslant D}$ satisfies the assumptions of \cite[Proposition 3.6]{Puig1991}. Thus the class $v$ is $N_G(D,e_D)$-stable, and $\Defres^{\,D}_{D/Q} v = v_Q$ for any non-trivial subgroup $Q$ of $D$.
By \cite[Corollary 8.5]{Bouc2006} and  \cite[Lemma 28.1]{Thevenaz1995}, there exists a unique isomorphism class of capped indecomposable endo-permutation $\ringO D$-module $V$ with determinant 1 (\emph{i.e.}, with a structure map that sends the group $D$ into $SL(V)$) such that $v= [k\otimes_\ringO V ]$. Since the defect group $D$ is abelian, the normaliser $N_G(D,e_D)$ controls the $e$-fusion in the group $G$ with respect to the maximal subpair $(D,e_D)$, so the triple $(D,e_D,V)$ is fusion-stable in $G$.

Thirdly, we suppose that the prime $p$ is odd and that the poset $\mathcal A_{\geqslant 2}(D)$  is connected. For any subgroup $Q$ of $D$, the class $v_Q\in\dade(N_D(Q)/Q)$ contains the source of a simple module for the $p$-nilpotent group $(S_Q\times C_{S_Q}(P))\rtimes N_D(Q)/Q$. Thus we know from \cite[Proposition 4.4]{BoltjeKuelshammer} that the class $v_Q$ lies in the torsion part $\dade_t(N_D(Q)/Q)$ of the Dade group $\dade(N_D(Q)/Q)$. By \cite[Theorem 1.1]{BoucThevenaz}, there is an exact sequence
\[
0 
\ \to\ 
\dade_t(D)
\ \to\ 
\varprojlim_{1\neq Q\leqslant D} \dade_t(N_D(Q)/Q)
\ \to\ 
\tilde H^0(\mathcal A_{\geqslant 2}(D))
\ \to\ 
0.
\]
By lemma \ref{lem:compatible sources}, the family $(v_Q)_{1\neq Q\leqslant D}$ lies in the direct limit
$\varprojlim_{1\neq Q\leqslant D} \dade_t(N_D(Q)/Q)$ of the above exact sequence. By assumption, the additive group $\tilde H^0(\mathcal A_{\geqslant 2}(D),\mathbb F_2)$ of locally constant $\mathbb F_2$-valued functions on $\mathcal A_{\geqslant 2}(D)$, modulo constant functions, is trivial. Thus there exists a unique class $v$ in the torsion Dade group $\dade_t(D)$ such that $\Defres^{\,D}_{N_D(Q)/Q} v = v_Q$ for any non-trivial subgroup $Q$ of $D$. As above, there is a unique capped indecomposable endo-permutation $\ringO D$-module $V$ with determinant 1 such that the reduction $k\otimes_\ringO V$ belongs to the class $v$.

Let $(R,e_R)$ be a subpair of $(D,e_D)$ and let $g\in G$ be such that $^g(R,e_R)\leqslant (D,e_D)$. Set $w=\Res^{D}_{R} v$ and $w'=\Res^{D^g}_{R} g^{-1}\cdot v$. For any non-trivial subgroup $Q$ of $R$, we have 
\begin{align*}
\Defres^{\,R}_{N_R(Q)/Q} w  
\ &=\ 
\Res^{\,N_D(Q)/Q}_{N_R(Q)/Q} v_Q 
\\ &=\ 
\Res^{\,N_{D^g}(Q)/Q}_{N_R(Q)/Q} g^{-1}\cdot v_{\hspace{.1mm}^gQ}
\ =\ 
\Defres^{\,R}_{N_R(Q)/Q} w'.  
\end{align*}
Then the injectivity of the deflation-restriction map $\dade_t(R) \to \varprojlim_{1\neq Q\leqslant R} \dade_t(N_R(Q))$ implies that $w=w'$. Let $W$ (\emph{resp.} $W'$) be a capped indecomposable direct summand of the restriction $\Res^{D}_{R} V$ (\emph{resp.} $w'=\Res^{D^g}_{R} g^{-1}V$). Since the prime $p$ is odd, the endo-permutation $\ringO R$-modules $W$ and $W'$ must have determinant 1; moreover, the reductions $k\otimes_\ringO W$ and $k\otimes_\ringO W'$ belong to the same class $w=w'\in\dade(R)$. Thus $W$ and $W'$ are isomorphic, and the triple $(D,e_D,V)$ is fusion-stable in the group $G$.
\end{proof}

For a general defect group $D$, the obstruction group $\tilde H^0(\mathcal A_{\geqslant 2}(D),\mathbb F_2)$ needs not be trivial. However, we know from the classification of finite simple groups that the $Z^*_p$-theorem is always true. This implies that Assumption \ref{hyp:gluing sources} is satisfied, at least when the centraliser $C_G(P)$ controls the $p$-fusion in the group $G$ (and not only the $e$-fusion). We do hope that a careful study of the direct image of the family $(v_Q)_{1\neq Q\leqslant D}$ in the obstruction group $\tilde H^0(\mathcal A_{\geqslant 2}(D),\mathbb F_2)$ will show that this direct image is always trivial. This would allow one to prove Theorem \ref{thm:stable equivalence} without any restriction on the defect group $D$.

%=================
\section{Obtaining a stable equivalence}
\label{sec:GlobalCase}

With all the conventions of the previous section, we now suppose that Assumption \ref{hyp:gluing sources} is satisfied. We identify $V$ with an $\ringO \Delta D$-module. By Lemma \ref{lem:UnicM}, there is a unique indecomposable Brauer-friendly $\ringO (G\times C_G(P))$-module $M$ with source triple $(\Delta D, e_D\otimes e_D^\bop, V)$ such that the slashed module $M\sla{\Delta P,e_P\otimes e_P^\bop}$ admits the $k(C_G(P)\times C_G(P))$-module $kC_G(P)e_P$ as a direct summand. 

\begin{lem}
\label{lem:global module local Morita}
Let $Q$ be a non-trivial subgroup of the defect group $D$. Then the slashed module $M\sla{\Delta Q,e_Q\otimes e_{PQ}}$ induces a Morita equivalence 
\[
kC_G(Q)e_Q \ \sim\ kC_G(PQ) e_{PQ}
\]
\end{lem}

\begin{proof}
The class $\Defres^D_{D/P} v\in\dade(D/P)$ is trivial, so the slashed module $M\sla{\Delta P,e_P\otimes e_P^\bop}$ is a $p$-permutation $k(C_G(P)\times C_G(P))$-module. Thus we may use, from now on, the slash construction that we have defined in Lemma \ref{lem:CanSlashMod}. For the sake of shortness, whenever $Q$ is a subgroup of the defect group $D$, we write 
\begin{gather*}
C(Q)=(C_G(Q)\times C_G(PQ))
\ \ ;\ \ 
N(Q)=C(Q)\Delta N_G(Q,e_Q)
\ \ ;\ \\ 
M\sla Q = M\sla{\Delta Q,e_Q\otimes e_{PQ}},
\end{gather*}
where the latter is a $kN(Q)$-module.
For any $Q\leqslant D$ and any $g\in G$ such that $^g(Q,e_Q)\leqslant (D,e_D)$, the uniqueness part of Lemma \ref{lem:CanSlashMod} implies that there is an isomorphism of $kN(Q)$-modules $M\sla Q \ \simeq \ (g^{-1},g^{-1})\cdot M\sla{\null^gQ}$.
Thus, up to replacing the subgroup $Q$ by a $G$-conjugate, we may suppose that the subpair $(Q,e_Q)$ is fully normalised in $(D,e_D)$, \emph{i.e.}, that the normaliser subpair $(N_D(Q),e_{N_D(Q)})$ is a maximal $e_Q$-subpair of the group $N_G(Q,e_Q)$. Similarly, for any two subgroups $Q\triangleleft R$ of $D$, the $kN(Q,R)$-modules $M\sla Q\sla R$ and $\Res^{N(R)}_{N(Q,R)} M\sla R$ are isomorphic.

By construction of $M$, we know that the $kC(P)$-module $M\sla P$ admits $kC_G(P)e_P$ as a direct summand. As a consequence, the Brauer quotient $M\sla P\sla Q$ admits the $kN(P,Q)$-module $kC_G(PQ)e_{PQ}$ as a direct summand. By the above remark on the transitivity of the slash construction, it follows that the slashed module $M\sla Q\sla P$ also admits $kC_G(PQ)e_{PQ}$ as a direct summand. Thus there exists an indecomposable direct summand $M_Q^0$ of the $kN(Q)$-module $M\sla Q$ such that the slashed module $M_Q^0\sla P$ admits the $kN(P,Q)$-module $kC_G(PQ)e_{PQ}$ as a direct summand. 

Let $(R,f)$ be a maximal $e_Q$-subpair of the group $N_G(Q,e_Q)$. Then the Brauer quotient $\Br_{(\Delta R,f\otimes f^\bop)}(kC_G(PQ)e_{PQ})\simeq kC_G(R)f$ is non-zero. By transitivity of the slash construction, it follows that the slashed module $M_Q^0\sla{\Delta R,f\otimes f^\bop}$ is non-zero. Moreover, a vertex subpair of $M_Q^0$ must be contained in a conjugate of the vertex subpair $(\Delta D,e_D\otimes e_D^\bop)$ of $M$. Thus $(\Delta R,f\otimes f^\bop)$ is a vertex subpair of $M_Q^0$. Assuming that the subpair $(Q,e_Q)$ is fully normalised in $(D,e_D)$, we deduce from Lemma \ref{lem:SlashMod} (iii) that $(\Delta N_D(Q),e_{N_D(Q)}\otimes e_{N_D(Q)}^\bop,V_Q)$ is a source triple of the indecomposable $kN(Q)$-module $M_Q^0$. Now it follows from Lemma \ref{lem:UnicM} and Theorem \ref{thm:local equivariant Morita equivalence} that the $kN(Q)$-module $M_Q^0$ induces an $N_G(Q,e_Q)/C_G(Q)$-equivariant Morita equivalence
\[
kC_G(Q)\bar e_Q \ \sim\ kC_G(PQ)\bar e_{PQ}.
\]
\hspace{\parindent}The next step uses descending induction on the order of the group $Q$ to prove that $M\sla Q = M_Q^0$. We know from the proof of Lemma \ref{lem:UnicM} that the slashed module $M\sla D$ is isomorphic to the indecomposable $kN(D)$-module $kC_G(D)e_D$, so $M\sla D=M_D^0$.
Then let $Q$ be a proper subgroup of $D$ and suppose that $M(R) = M_R^0$ for any $p$-group $R$ such that $Q<R\leqslant D$. We consider a Krull-Schmidt decomposition
\[
M\sla Q = M_Q^{0}\oplus\ldots\oplus M_Q^{n},
\]
of the $kN(Q)$-module $M\sla Q$, and we suppose that $n\geqslant 1$. Let $(R,f)$ be a vertex subpair of the $kN(Q)$-module $M_Q^1$. Once again, we may assume that the subpair $(Q,e_Q)$ is fully normalised in $(D,e_D)$. We may suppose that $(R,f)$ is contained in the maximal $(e_Q\otimes e_Q^\bop)$-subpair $((C_D(Q)\times C_D(Q))\Delta N_D(Q), e_{N_D(Q)}\otimes e_{N_D(Q)}^\bop)$.
By Lemma \ref{lem:SlashMod} (iii), the subpair $(R,f)$ must be contained in a $(G\times C_G(P))$-conjugate of the vertex subpair $(\Delta D,e_D\otimes e_D^\bop)$ of $M$. Thus we have $\Br_{(R,f)}(kGe)\neq 0$. 
Since the subpair $(\Delta Q,e_Q\otimes e_Q^\bop)$ is normalised by $(R,f)$ and $\Br_{(\Delta Q,e_Q\otimes e_Q^\bop)}(kG\bar e)\simeq kC_G(Q)\bar e_Q$, it follows that $\Br_{(R,f)}(kC_G(Q)\bar e_Q)\neq 0$. 
So the subpair $(R,f)$ is contained in a $(C_G(Q)\times C_G(Q))\Delta N_G(Q,e_Q)$-conjugate of the vertex subpair $(\Delta N_D(Q),e_{N_D(Q)}\otimes e_{N_D(Q)}^\bop)$ of the indecomposable $k(C_G(Q)\times C_G(Q))\Delta N_G(Q,e_Q)$-module $kC_G(Q)e_Q$. 
Moreover the subgroup $N(P,Q)$ controls the $(e_Q\otimes e_Q^\bop)$-fusion in the group $(C_G(Q)\times C_G(Q))\Delta N_G(Q,e_Q)$. So $(R,f)$ is contained in an $N(P,Q)$-conjugate of 
\linebreak
$(\Delta N_D(Q),e_{N_D(Q)}\otimes e_{N_D(Q)}^\bop)$. We may choose $(R,f)=(\Delta R',e_{R'}\otimes e_{PR'}^\bop)$ for some subgroup $R'$ of $N_D(Q)$.
If $Q<R'$, then we obtain
\[
M\sla{R'} \simeq M\sla Q\sla{R'} = M_Q^{0}\sla{R'}\oplus\ldots\oplus M_Q^{n}\sla{R'},
\]
where at least the direct summands $M_Q^0\sla{R'}$ and $M_Q^1\sla{R'}$ are non-zero. This contradicts the indecomposability of the $kC(R')$-module $M\sla{R'}= M_{R'}^0$. If $Q=R'$, then Lemma \ref{lem:DirSumdLift} implies that the $k(G\times C_G(P))$-module $M$ has an indecomposable direct summand with vertex subpair $(\Delta Q,e_Q\otimes e_{PQ}^\bop)$, another contradiction. So the lemma is proven.
\end{proof}

For the reader's convenience, we quote \cite[Theorem 1.1]{LinckelmannUnpublished}, which is not published yet. We slightly adapt the notations to fit those of the present chapter.

\begin{thm*}[Linckelmann]
Let $A$, $B$ be (almost) source algebras of blocks of finite group algebras over $\ringO$ having a common defect group $D$ and the same fusion system $\catf$ on $D$. Let $V$ be an $\catf$-stable indecomposable endo-permutation $\ringO D$-module with vertex $D$, viewed as an $\ringO \Delta D$-module through the canonical isomorphism $\Delta D\simeq D$. Let $M$ be an indecomposable direct summand of the $(A,B)$-bimodule
\[
A\otimes_{\ringO D} \Ind_{\Delta D}^{D\times D} V \otimes _{\ringO D} B
\]
Suppose that $M\otimes_B M^*\neq 0$. Then, for any non-trivial fully $\catf$-centralised subgroup $Q$ of $D$, there is a canonical $(\Br_Q(A),\Br_Q(B))$-bimodule $M\sla{\Delta Q}$ satisfying $\Endom_k(M\sla{\Delta Q}) \simeq \Br_{\Delta Q}(\Endom_\ringO(M))$. Moreover, if for all non-trivial fully $\catf$-centralised subgroups $Q$ of $D$ the bimodule $M\sla{\Delta Q}$ induces a Morita equivalence between $\Br_{\Delta Q}(A)$ and $\Br_{\Delta Q}(B)$, then $M$ and its dual $M^*$ induce a stable equivalence of Morita type between $A$ and $B$.
\end{thm*}

We now have all the tools that we need to prove our main result.

\begin{proof}[Proof of Theorem \ref{thm:stable equivalence}.]
Let $i\in(\ringO Ge)^D$ be a source idempotent of the block $e$ such that $\bar e_D\br_D(i)\neq0$, and let $i_P\in(\ringO C_G(P)e_P)^D$ be a source idempotent of the block $e_P$ such that $\bar e_D\br_D(i_P)\neq0$. Set $A=i\ringO Gi$ and $B=i_P\ringO C_G(P)i_P$. Then $iMi_P$ is an indecomposable direct summand of the $(A,B)$-bimodule $A\otimes_{kD}\Ind_{\Delta D}^{D\times D} V\otimes_{kD} B$, where $V$ is an endo-permutation $\ringO D$-module that is fusion-stable for the common fusion system of the source algebras $A$ and $B$ on the defect group $D$. Moreover, by Lemma \ref{lem:global module local Morita}, the slashed module $iMi_P\sla{\Delta Q}$ induces a Morita equivalence $\Br_{\Delta Q}(A)\sim \Br_{\Delta Q}(B)$ for any subgroup $Q$ of the defect group $D$.
Then Linckelmann's theorem asserts that the $(A,B)$-bimodule $iMi_P$ induces a stable equivalence $A\sim B$. In terms of block algebras, this means exactly that the $(\ringO Ge,\ringO C_G(P)e_P)$-bimodule $M$ induces a stable equivalence
\[
\ringO Ge \ \sim \ \ringO C_G(P)e_P.
\]
\vspace{-10mm}

\end{proof}

\end{document}